\documentclass[reqno,final]{amsart}
\usepackage{natbib}  %Nature-like bibliography
\usepackage{fancyhdr} %Headers
\usepackage{color} %Color definition
\usepackage{hyperref} %Internal and external links
\usepackage{graphicx} %Graphics inclusion

\usepackage{amssymb}
\usepackage{tikz}
\usepackage{booktabs}

\definecolor{aleacolor}{rgb}{0.16,0.59,0.78}

\hypersetup{
breaklinks,
colorlinks=true,
linkcolor=aleacolor,
urlcolor=aleacolor,
citecolor=aleacolor}

\renewcommand{\cite}{\citet}

\theoremstyle{plain}
\newtheorem{theorem}{Theorem}[section]                                          
                          
\newtheorem{lemma}[theorem]{Lemma}

\theoremstyle{definition}

\theoremstyle{remark}

\makeatletter \@addtoreset{equation}{section} \makeatother

\newcommand{\cor}{\operatorname{Corr}}
\newcommand{\mean}{\mathbb{E}}
\newcommand{\var}{\mathbb{V}}
\newcommand{\prob}{\mathbb{P}}

\newcommand{\ind}{\mathbf{1}} % here mathbf instead of mathbbold
\newcommand{\unif}{\textsf{Unif}}
\newcommand{\expdist}{\textsf{Exp}}
\newcommand{\geo}{\textsf{Geo}}

\begin{document}

\title[Minimum correlation for any
bivariate Geometric distribution]{Minimum correlation for any
bivariate Geometric distribution}

\author{Mark Huber}
\author{Nevena Mari\'c}

\address{Claremont McKenna College \newline
850 Columbia Avenue,\newline
Claremont, CA  91711, USA.}

\address{University of Missouri-St. Louis\newline
One University Boulevard,\newline
St. Louis, MO 63121, USA.}

\email{mhuber@cmc.edu, maricn@umsl.edu}
\urladdr{\url{http://www.cmc.edu/pages/faculty/MHuber/},\url{http://www.cs.umsl.edu/~maric}}

\thanks{Research supported by National Science Foundation and UMSL CAS Research Award}

\subjclass[2000]{60E05, 62H20.} 
\keywords{Geometric distribution, Minimum correlation.}

\begin{abstract}
  Consider a bivariate Geometric random variable where the first component
has parameter $p_1$ and the second  parameter $p_2$.  It is not possible to 
make the correlation between the marginals equal to -1.
Here the properties of
this minimum correlation are studied both numerically and analytically.  
It is shown that the minimum correlation can be computed exactly in time
$O(p_1^{-1} \ln(p_2^{-1}) + p_2^{-1} \ln(p_1^{-1}))$.  
One method for generating a bivariate geometric with
target correlation requires computing this minimum correlation.
The minimum correlation is
shown to be nonmonotonic in $p_1$ and $p_2$, moreover, the partial
derivatives are not continuous.  For $p_1 = p_2$, these discontinuities
are characterized completely and shown to lie near (1~- roots of 1/2).
In addition, we construct analytical bounds on the minimum correlation.

\end{abstract}

\maketitle

%In this latex style the {\tt $\backslash$cite} command creates
%citations using the name of the author and the date of publishing, for
%example \cite{MR1557756}.  We can also group citations as
%\cite{MR1557756,MR1557978,MR1557852}. If the citation is between
%parentheses, use the command {\tt $\backslash$citealp} to avoid the
%double parentheses (for example \citealp{MR0050410}).  It is very
%important to attach the file \emph{.bib} with all references (see {\tt
%  example.bib}).

%To create graphics, include images either in ps or in eps formats
%(Figure~\ref{fig1}), or use graphics package as illustrated in
%Figure~\ref{fig2}.  

%Tagged and untagged formulas are presented in the next section, where
%we illustrate the use of the commands {\tt $\backslash$section}, {\tt
%  $\backslash$subsections}, and the environments {\tt
%  $\backslash$begin$\{$theorem$\}$}, {\tt
%  $\backslash$end$\{$theorem$\}$}.

\section{Introduction}

%In this article 
We investigate the minimum attainable correlation between two 
Geometric random variables. 
%The topic of attainable correlations in 
%multivariate distributions has somehow been neglected in probabilistic 
%community. 
%Most of the students graduate with a firm idea that correlations 
%between two random variables are always in $[-1,1]$, where $-1$ 
%corresponds to minimum and $+1$ to maximum possible correlation. 
Most students graduate believing that any correlation in $[-1,1]$ 
is attainable by a bivariate distribution.
That, of course, is not true, except for distributions with symmetric 
support like Normal and Uniform (see \cite{Moran1967a}). 
The consequence is that, in data analysis,  empirical correlation is 
often misinterpreted, and compared to $-1$ and $1$ instead to the 
theoretical bounds.  See~\cite{denuitd2003} and~\cite{shihh1992}
for a discussion.
Therefore, attainable correlation is crucial 
information about a multivariate distribution. Still, there is much 
more unknown than known facts in this field, especially in higher dimensions. 
In bivariate case, minimum correlation for several important 
distributional examples is analyzed in \cite{Conway1979} and 
\cite{dukicm2013} (and references therein). The purpose of the 
present paper is to fill the gap in this subject concerning one 
of the most important discrete cases--the Geometric distribution. 

\par Say that $X$ has a Geometric distribution with parameter $p$ ($0<p \leq 1$) and
write $X \sim \geo(p)$, 
 if
for all $i \in \{0,1,2,\ldots\}$, $\prob(X = i) = p(1-p)^i$. If one has a coin with probability $p$ of heads, then $X \sim \geo(p)$ 
represents the number of tails flipped before obtaining a heads.

For $(p_1,p_2) \in  (0,1]^2$, let
\[
\rho_{-}(p_1,p_2) = \min\{\cor(X_1,X_2):X_1 \sim \geo(p_1),X_2 \sim \geo(p_2)\}.
\]
When $p_1 = p_2 = p$, Figure~\ref{FIG:graph} shows a graph of this
minimum correlation as a function of $p$.
Several properties are immediately apparent.  First, the correlation is {\em not
a monotonic function of $p$}.  In addition, there are points of 
discontinuity in the derivative of the 
graph.  These phenomena are explained
in Section~\ref{SEC:properties}.

In Section~\ref{SEC:compute}
it is shown that the value of $\rho_{-}(p_1,p_2)$ can be found exactly
in time $O(p_1^{-1}\ln(p_1^{-1}) + p_2^{-1}\ln(p_2^{-1}))$.  In addition,
upper and lower bounds for this function are computed.

\begin{figure}[ht]
\includegraphics[width=5in]{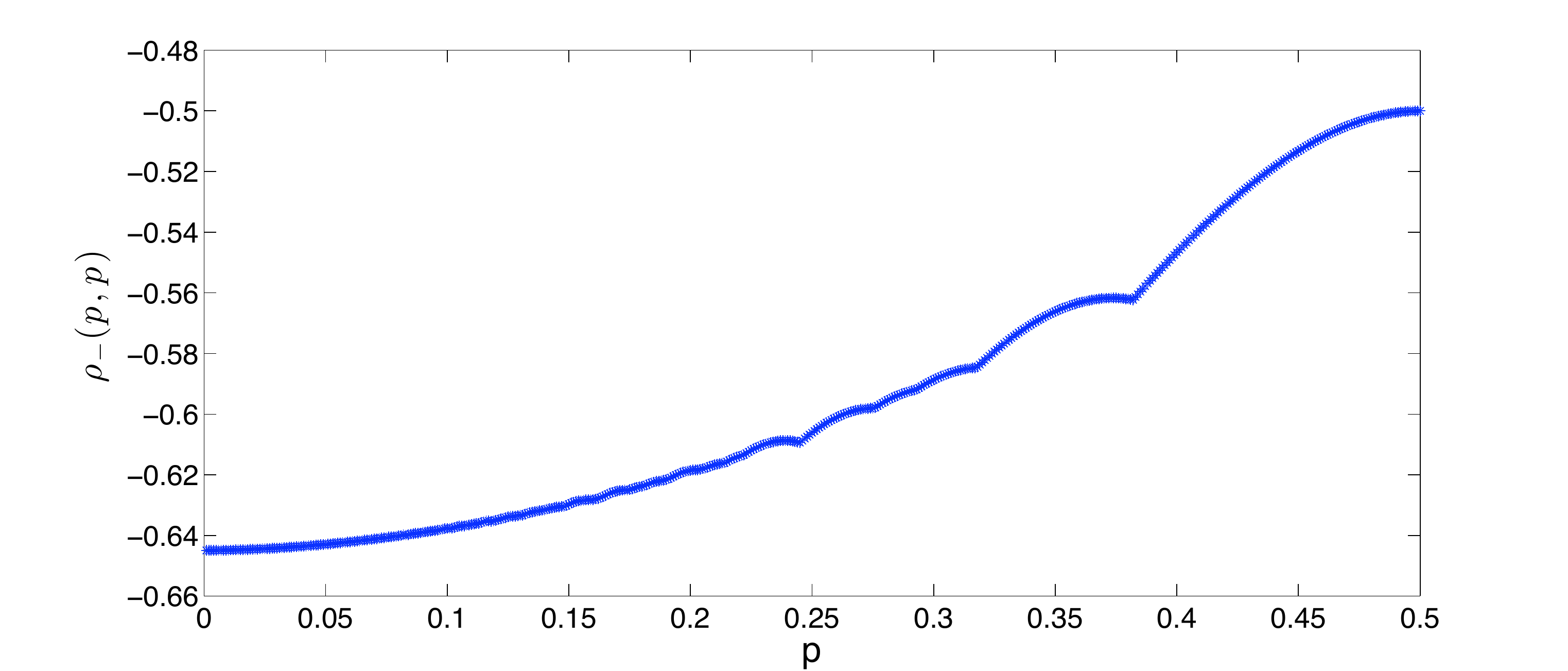} 
\caption{The minimum correlation $\rho_{-}(p,p)$ for $p<1/2$. When $p\geq 1/2$ the minimum correlation is simply equal to $p-1$.}
\label{FIG:graph}
\end{figure}

To understand $\rho_{-}$, first consider
the inverse transform method for generating a random variate with
a specified cdf (cumulative distribution function) $F$.  Define the pseudoinverse of the cdf as
\begin{equation}
F^{-1}(u) = \inf\{x:F(x) \geq u\}.
\end{equation}
When $U$ is uniform over the interval $[0,1]$ (write $U \sim \unif([0,1])$),
$F^{-1}(U)$ is a random variable with cdf $F$ 
(see for instance p.~28 of~\cite{devroye1986}).  Since $U$ and $1 - U$ 
have the same distribution, both can be used in the inverse transform
method.  The random variables $U$ and $1 - U$ are {\em antithetic} random
variables.  
%Since $\cor(U,U) = 1$ and $\cor(U,1-U) = -1$, so these 
%represent an easy way to get minimum and maximum correlation when the 
%marginals are uniform random variables.

We will use the notation $X \sim Y$ when $X$ has the same probability distribution as $Y$. The following result comes from work of \cite{frechet1951}
and \cite{hoeffding1940}.

\begin{lemma}[Fr\'echet-Hoeffding bound]
\label{THM:fhbound}
For $X_1$ with cdf $F_1$ and $X_2$ with cdf $F_2$,
and $U \sim \unif([0,1])$:
\[
\cor(F_1^{-1}(U),F_2^{-1}(1 - U)) \leq \cor(X_1,X_2) \leq 
  \cor(F_1^{-1}(U),F_2^{-1}(U)).
\]
Conversely, if $\cor(X_1,X_2)$ equals the minimum correlation then
it holds that $(X_1,X_2) \sim (F_1^{-1}(U),F_2^{-1}(1-U))$.  For correlation equal
to
the maximum value, $(X_1,X_2) \sim (F_1^{-1}(U),F_2^{-1}(U))$. 
\end{lemma}

 In other words, the maximum correlation between $X_1$ and $X_2$ is 
achieved when the same uniform is used in the inverse transform method
to generate both.  The minimum correlation between $X_1$ and $X_2$ is
achieved when antithetic random variates are used in the inverse transform
method. In the literature on dependence and copulas (see 
for instance \citealp{nelsen2007} and \citealp{denuitd2003}) $(F_1^{-1}(U),F_2^{-1}(U))$ and 
$(F_1^{-1}(U),F_2^{-1}(1-U))$ are known as the
{\em comonotonic} and {\em countermonotonic} vectors, respectively.

For $X \sim \geo(p)$, 
the expectation and variance are well 
known: $\mean(X)= (1-p)/p$ and $\var(X)= (1-p)/p^2$.  The cdf is
\[
F_p(a)=\prob(X \leq a) = p + p(1 - p) + \cdots p(1 - p)^{a} = 1 - (1 - p)^{a+1}. 
%1 - \frac{p(1-p)^{a+1}}{1 - (1 - p)} = 
\]

\begin{lemma}
The pseudoinverse $F_p^{-1}$ of $F_p$ is
\[
F_p^{-1}(u) = \sum_{n = 1}^\infty \ind(1 - (1 - p)^n \leq u < 1 - (1 - p)^{n+1}).
\]
\end{lemma}

[Here $\ind(\text{expression})$ is the indicator function that evaluates
to 1 when the Boolean expression in the argument is true, and is 0 otherwise.]

\begin{proof}
As the cdf of $X$ is $1 - (1 - p)^{a+1}$, 
for $u \in [1 - (1 - p)^n \leq u < 1 - (1 - p)^{n+1}]$,
it holds that $\prob(X \leq n) \geq u$ and $\prob(X \leq n - 1) < u$.  
\end{proof}

\paragraph{Prior Work}

Several authors have studied the construction of bivariate geometric
distributions.  
%Any particular family gives an upper bound on the 
%minimum correlation.  For example, 
%Guldberg~\cite{guldberg1934} gave a bivarate geometric where
%\[
%\prob(X_1 = m,X_2 = n) = \binom{m + n}{m} p^{m+n}(1 - 2p).
%\]
\cite{downton1970} created such a distribution as a means to
create a bivariate exponential for 
reliability applications
where two processes are receiving shocks in a memoryless correlated 
fashion.
\cite{hawkes1972} generalized Downton's family as follows.
Consider a bivariate Bernoulli distribution
$(A,B)$ where for all $i$ and $j$ in $\{0,1\}$:
\[
\prob(A = i,B = j) = p_{ij}.
\]
Then if $(A_i,B_i)$ are an iid sequence of draws from this distribution
 for $i \in \{1,2,3,\ldots\}$, let $X_1 = \min\{i:A_{i+1} = 1\}$,
$X_2 = \min\{i:B_{i+1} = 1\}$.  It is easy to show that this gives
$X_1 \sim \geo(p_{10} + p_{11}),$ $X_2 \sim \geo(p_{01} + p_{11})$.  

\cite{marshallo1985} then showed that 
the geometrics obtained in this fashion have a minimum correlation 
of at least $-1/4$.

\cite{paulsonu1972} built a 
bivariate distribution by taking advantage of a recursive formulation
of the geometric from~\cite{uppulurifs1967}.  They do not analyze
the minimum correlation, only showing that their family of distributions
is not rich enough to include the case that the components are
independent.

In~\cite{dukicm2013} (and see also~\citealp{hubertoappeard}), it is shown how
to simulate a bivariate Geometric distribution that attains 
any value between the maximum and minimum correlation, although these
methods require knowledge of the maximum and minimum correlation.

Therefore our 
first main result concerns computation of the minimum correlation.
Since the bivariate geometric distribution has infinite support, it is important
to note the minimum correlation can be computed relatively quickly.
\begin{theorem}
\label{THM:first}
  The minimum correlation between $X_1 \sim \geo(p_1)$ and 
  $X_2 \sim \geo(p_2)$ can be computed in time 
  $O(p_1^{-1}\ln(p_2^{-1}) + p_2^{-1}\ln(p_1^{-1}))$.
\end{theorem}

Our second main result is a proof of certain properties of the 
function $\rho_{-}(p,p)$.
\begin{theorem}
\label{THM:second}
Let $\rho_{-}(p)$ be the minimum correlation achieved between $X_1$ 
and $X_2$ where both are $\geo(p)$.  Then the following is true.
\begin{enumerate}
\item{There is an infinite
number of points where $(d/dp)\rho_{-}(p)$ is discontinuous.}
\item{The points where the discontinuities occur are 
near to  $(1$~- roots of $1/2)$.}
%\item{At each discontinuity, the derivative goes from being negative
%to being positive.}
\item{The function is upper and lower bounded by:
\[
g(p) - p \leq
  \rho_{-}(p) \leq g(p)
\]
where
\[
g(p) = \frac{p^2}{[\ln(1-p)]^2} \cdot \frac{1}{1-p}\left(2 - \frac{\pi^2}{6}
 \right) - (1 - p).
\]
}
\end{enumerate}
\end{theorem}

To bound $\cor(F_p^{-1}(U), F_p^{-1}(1-U))$, the minimum correlation, the key
is computing 
$\mean(F_p^{-1}(U)F_p^{-1}(1-U))$.   Section~\ref{SEC:compute}
looks at finding this quantity for various values of $p$. Some computational details are left for the Appendix, Section~\ref{SEC:appendix}.
Section~\ref{SEC:properties} then proves an upper and lower bound on
the $\rho_{-}(p)$ function, as well as the asymptotic behavior of the 
``bumps'' in the function.

\section{Computing the minimum correlation}
\label{SEC:compute}

For simplicity consider first the case that $p = p_1 = p_2$.

For any bivariate random variables with the same marginal distributions, the maximum correlation is always 1.
More interesting is the minimum
correlation.  For geometric marginals, the minimum correlation
is markedly different 
when $p < 1/2$ and when $p \geq 1/2$.

\begin{lemma}
  Let $\rho_{-}(p)$ be the minimum correlation achievable between
  $X_1$ and $X_2$ where both are $\geo(p)$.
  It is possible to compute $\rho_{-}(p)$ in 
  $O(p^{-1} \ln(p^{-1}))$ steps.
\end{lemma}

\begin{proof}
  As in the introduction, let $U \sim \unif([0,1])$, $X_1 = F_p^{-1}(U)$
  and $X_2 = F_p^{-1}(1 - U).$

  $\bullet$ Consider the $p \geq 1/2$ case.  Then either
  $U$ or $1 - U$ falls in the interval $[0,p]$ so either $X_1$ or $X_2$ is 0.
  Hence $\mean(F_p^{-1}(U)F^{-1}_p(1-U)) = 0$ and
  \[
  \rho_{-}(p) = \frac{0 - [(1 - p)/p]^2}{(1-p)/p^2} = p - 1.
  \]

  $\bullet$ Next suppose $p < 1/2$.  As in the $p \geq 1/2$ case, if either
  $U$ or $1 - U$ falls in $[0,p]$, then $X_1 X_2 = 0$ and so consider when
  $U \in [p,1-p]$.  

  Let $q = 1 - p$, $\alpha_i = 1 - q^i$, and $\beta_i = q^i$.  With this
  notation, $F_p(i) = \alpha_{i+1}$, and the pseudoinverse becomes
  \[
  F^{-1}_p(u) = \sum_{i=1}^\infty \ind(U \in [\alpha_i,\alpha_{i+1})).
  \]

  Note that 
  \[
  p = \alpha_1 < \alpha_2 < \cdots < \alpha_c \leq 1 - p,
  \]
  where $c = \lfloor \ln(p)/\ln(1 - p) \rfloor = 
         \lfloor \log_{1 - p}(p) \rfloor.$
  
  When $U \in [\alpha_{i},\alpha_{i+1})$, $X_1 = i$.  At the same time,
  when $1 - U \in [\alpha_{i},\alpha_{i+1})$, $X_2 = i$.  Hence there are at
  most $2c$ breakpoints changing the value of $X_1$ or $X_2$.  Therefore
  there are at most $2c$ different values of $(X_1,X_2)$ where one
  of the variables is not 0.  This makes it possible to compute
  $\mean(X_1X_2)$ in $O(c) = O(p^{-1}\ln(p^{-1}))$ time. For more details see the Appendix.
\end{proof}

\paragraph{Example: ${\mathbf{p=1/4}}$}
As an example of how this can be used to calculate the minimum
correlation, consider the case when $p = 1/4$. 

Here $c=\lfloor \ln(1/4)/\ln(3/4) \rfloor = 4$, 
and so the $\alpha_i$ and $\beta_i$ values for interval
$[1/4,3/4]$ become
\begin{center}
\begin{tabular}{rllll}
  & \multicolumn{4}{c}{$i$} \\
  & 1 & 2 & 3 & 4 \\
\midrule
$\alpha_i$ & $1/4 = 64/256$ & $112/256$ & $148/256$ & $175/256$ \\
$\beta_i$ & $3/4 = 192/256$ &  $144/256$ & $108/256$ & $81/256$ \\
\end{tabular}
\end{center}
Ordering the $\alpha_i$ and $\beta_i$ divides $[1/4,3/4]$ into
seven pieces:
\[
(x_1,x_2,\ldots,x_8) = 
 \left(\frac{65}{256},\frac{81}{256},\frac{108}{256},\frac{112}{256},
\frac{144}{256},\frac{148}{256},\frac{175}{256},\frac{192}{256}\right).
\]
The seven intervals are then
\begin{center}
\begin{tabular}{rlllllll}
Interval    & $[x_1,x_2]$ & $[x_2,x_3]$ & $[x_3,x_4]$ & $[x_5,x_6]$ & $[x_6,x_7]$ 
 & $[x_1,x_2]$ & $[x_1,x_2]$ \\
 \midrule
$(X_1,X_2)$ & $(1,4)$ & $(1,3)$ & $(1,2)$ & $(2,2)$ & $(2,1)$ & 
 $(3,1)$ & $(4,1)$
\end{tabular}
\end{center}
Hence
\[
\mean(X_1X_2) = 1 \cdot 4 \cdot \frac{81 - 65}{256} + 
  1 \cdot 3 \cdot \frac{108 - 81}{256} + \cdots
  + 4 \cdot 1 \cdot \frac{192 - 175}{256} = \frac{442}{256}
 \approx 1.7266,
\]
which gives a minimum correlation of $\rho_{-}(1/4) = -1862/3072 = -0.606.$

\begin{lemma}
  Let $\rho_{-}(p_1,p_2)$ be the minimum correlation achievable between
  $X_1 \sim \geo(p_1)$ and $X_2 \sim \geo(p_2)$.  Then
  it is possible to compute $\rho_{-}(p_1,p_2)$ in 
  $O(p_1^{-1} \ln(p_2^{-1}) + p_2^{-1} \ln(p_1^{-1}))$ steps.
\end{lemma}

\begin{proof}
  The proof is essentially the same as for the previous lemma.  Since
  $\mean[X_1]$, $\mean[X_2]$, $\var(X_1)$, and $\var(X_2)$ are easy
  to calculate, the difficult part is finding $\mean[X_1 X_2]$ using
  antithetic random variables.  
  \par Let $\chi_1(u)=\lfloor \log_{1 - p_1}(1 - u)\rfloor$ and $\chi_2(u) = \lfloor \log_{1 - p_2}(u)\rfloor$ for $u \in [0,1]$. 
  Then note $X_1 = \chi_1(U)$
  and $X_2 = \chi_2(U)$, so
  \[
  \mean[X_1 X_2] = \int_0^1 \chi_1(u) \chi_2(u) \ du.
  \]
  Find the integral by breaking it into a sum, since 
  $\chi_1(u)$ and $\chi_2(u)$ are both step functions.

  When $p_1 + p_2 \geq 1$, then one of the $X_1$ and $X_2$ must be zero.
  Otherwise let $\alpha_i = 1 - (1 - p_1)^i$ for $i$ from 1 to  
  $d_2 = \lfloor \ln(p_2)/\ln(1 - p_1)\rfloor$.  Similarly, set
  $\beta_i = (1 - p_2)^i$ for $i$ from 1 to 
  $d_1 = \lfloor \ln(p_1)/\ln(1 - p_2)\rfloor$.  Note 
  $d_1 + d_2 = O(p_1^{-1} \ln(p_2^{-1}) + p_2^{-1} \ln(p_1^{-1}))$ and
  that the $\{\alpha_i\}$ and $\{\beta_j\}$ values can be merged and 
  sorted in linear time.
\end{proof}

Since for $X \sim \geo(p)$, $\mean[X] = O(p_1^{-1})$, this 
proves Theorem~\ref{THM:first}.

\section{Properties of the minimum correlation}
\label{SEC:properties}

In this section, the discontinuities of the partial derivatives of 
the $\rho_{-}(p_1,p_2)$ function are determined.

%For simplicity let's suppose that $p_1 \leq p_2$. 
Recall that $\mean[X_1 X_2]$ is computed by breaking the interval
$[0,1]$ into subintervals using $0 \leq s_1 \leq s_2 \leq s_3 \leq \cdots s_n \leq 1$
where $(s_1,\ldots,s_n)$ are the sorted values (order statistics) of 
the $\{\alpha_i\}$ and $\{\beta_j\}$.  In particular $s_1=\alpha_1$ and $s_n=\beta_1$. Also, for convenience we will set $s_0=0$.
Let $f_1(m) = \max \{i:\alpha_i \leq s_m\}$,\ 
     %f_2(m) = \min\{j:\beta_j \geq s_m\}$ 
     {$f_2(m) = \max\{j:\beta_j \geq s_{m+1}\}$ }
     so 
that for all $u \in (s_m,s_{m+1})$, $(\chi_1(u), \chi_2(u)) = (f_1(m),f_2(m))$.
In this interval form:
\begin{equation}
\label{EQN:interval}
\mean[X_1 X_2] = \sum_{m=1}^{n-1} (s_{m+1} - s_m)f_1(m) f_2(m).
\end{equation}

\begin{lemma}
Fix $p_2$, and let 
$\bar p_1$ be a value  where there exists $i$ and $j$ such that
$\alpha_i = \beta_j$.  Then $\partial \rho_{-}/\partial p_1$ 
has a discontinuity at $\bar p_1$.
\end{lemma}

\begin{proof}
Since 
$\rho_{-}(p_1,p_2) = (\mean[XY] - \mean[X]\mean[Y])/\sqrt{\var(X_1)\var(X_2)}$
and $\mean[X]$ and $\var(X_1)$ are analytic in $p_1$ for $p_1 \in (0,1]$,
it suffices to show that $\partial \mean[XY]/\partial p_1$ is discontinuous
at $\bar p_1$.

Each $\alpha_\ell$ is the left endpoint of one subinterval, and the right
endpoint of another.  Hence for each $\ell$ there is an integer $m(\ell)$ such
that $\alpha_\ell = s_{m(\ell)}$.  Note that when 
$s_{m(\ell) - 1} < s_{m(\ell)} = \alpha_\ell < s_{m(\ell) + 1}$, a small
change in $\alpha_\ell$ does not change the interval structure.  
That means $f_1(m(\ell))$ and $f_1(m(\ell - 1))$ are 
constant under small changes in $\alpha_\ell$.
Only two terms in $\mean[X_1 X_2]$ depend on $\alpha_\ell$,
so $\partial \mean[X_1 X_2]/\partial \alpha_\ell$ is
\[
\frac{\partial}{\partial \alpha_\ell}
   \left[(s_{m(\ell) + 1} - \alpha_\ell)f_1(m(\ell))f_2(m(\ell)) + 
   (\alpha_\ell - s_{m(\ell) - 1})]f_1(m(\ell) - 1)f_2(m(\ell) - 1)\right]
\]
and since $s_{m(\ell)  + 1}$ and $s_{m(\ell) - 1}$ do not depend on $\alpha_\ell$:
\[
\frac{\partial \mean[X_1 X_2]}{\partial \alpha_\ell}
 = -f_1(m(\ell)) f_2(m(\ell)) + f_1(m(\ell) - 1) f_2(m(\ell) - 1).
\]
This holds for all $\ell$.  The chain rule then gives
\begin{align*}
\label{EQN:expderiv}
\frac{\partial \mean[X_1 X_2]}{\partial p_1}
 &= \sum_{\ell } - \frac{\partial \alpha_\ell}{\partial p_1} f_1(m(\ell)) f_2(m(\ell))
  + \frac{\partial \alpha_\ell}{\partial p_1} f_1(m(\ell) - 1) f_2(m(\ell) - 1).
\end{align*}

Since $s_{m(\ell)} = \alpha_\ell$, $f_1(m(\ell)) = \ell$ and
$f_1(m(\ell) - 1) = \ell - 1$.
Also, we know that $f_2(m(\ell)) = f_2(m(\ell) - 1)$ since the boundary between
the $m$ and $m - 1$ intervals is $\alpha_\ell$.  Hence
\begin{equation}
\label{EQN:below}
\frac{\partial \mean[X_1 X_2]}{\partial p_1} 
 = -\sum_\ell \frac{\partial \alpha_\ell}{\partial p_1} f_2(m(\ell)).
\end{equation}

So now consider $p_1$ only slightly smaller than $\bar p_1$.  Then
$\alpha_i < \beta_j$, and if $p_1$ is close enough to $\bar p_1$, then
$s_{m(i) + 1} = \beta_j$.  As $p_1$ increases past $\bar p_1$, 
$\alpha_i$ increases past $\beta_j$.  Then gives 
$f_2(m(i))$ a discontinuity, as now the situation is 
$\beta_j < \alpha_{i} = s_{m(i)} < \beta_{j-1}$.
So $f_2(m(i))$ jumps from $j$ for $p_1$ arbitrarily close to but smaller than
$\bar p_1$,
to $j - 1$ for $p_1$ arbitrarily close to but larger than $\bar p_1$.

Note that  {$\partial \alpha_\ell/\partial p_1 > 0$} for all $\ell$.  So   
there might
be other $\{i',j'\}$ pairs where $\alpha_{i'} = \beta_{j'}$, but this only
makes the discontinuous jump larger.

Hence $\partial \mean[X_1X_2]/\partial p_1$ has a discontinuous jump
at every $p_1$ value where there is at least one $\alpha_i = \beta_j$.
\end{proof}

Of course by symmetry a similar result holds for $p_2$.  A similar result
also holds for $\rho_{-}(p) = \rho_{-}(p,p)$.

\begin{lemma}
When there is an $\{i,j\}$ pair such that 
$1 - (1 - \bar p)^i = (1 - \bar p)^j,$ the derivative of 
$\rho_{-}(p)$ is discontinuous at $\bar p$.
\end{lemma}

\begin{proof}
The proof is similar to that of the previous lemma.
\end{proof}

Consider the solutions to the equation of 
the previous Lemma.  For $x = (1 - \bar p)$, discontinuities
occur at the solutions to equations of the form 
\begin{equation}
\label{EQN:discontinuities}
x^j + x^i = 1.
\end{equation}

One simple family of 
solutions is all roots of $1/2$.  That is, setting $j = i$ and
$x = (1/2)^{1/i}$ gives a solution to~\eqref{EQN:discontinuities}.

The next set of solutions comes from $j = i + 1$, giving the equation
$x^i(1 + x) = 1$.  Since the solutions have $x$ close to 1, $1 + x$ is
close to $2$ and $x^i$ is close to $1/2$.  Since $1 + x$ is slightly
smaller than $2$, the solution $x$ is slightly larger than
$(1/2)^{1/i}$.

More generally, for any fixed $c$, a family of solutions is found with
$j = i + c$, with solution $x$ that is close to $(1/2)^{1/i}$.  The
following lemma makes this notion of closeness precise.

\begin{lemma}
The unique positive solution to $x^i(1 + x^c) = 1$ lies in the interval
\[
\left((1/2)^{1/i},(1/2)^{1/(i+c)}\right)
\]
for $i$ and $c$ positive.
\end{lemma}

\begin{proof}The function 
$f(x) = x^i(1 + x^c)$ is continuous in $x$ for $i$ and 
$c$ positive.  Note
\begin{align*}
f((1/2)^{1/i}) &= (1/2)(1 + (1/2)^{c/i}) < (1/2)(2) < 1 \\
f((1/2)^{1/(i+c)}) &= (1/2)^{i/(i+c)}(1 + (1/2)^{c/(i+c)}) = 
  (1/2)^{i/(i+c)} + (1/2) > 1.
\end{align*}
Hence the Intermediate Value Theorem guarantees a solution to 
$f(x) = 1$ for $x$ inside the interval.
\end{proof}

\section{Bounding $\mathbf \rho_{-}(p)$}

Using the antithetic generation of $X_1$ and $X_2$, it is 
possible to obtain bounds on $\rho_{-}(p_1,p_2)$.

\begin{lemma}
The minimum correlation satisfies
\[
\rho_{-}(p_1,p_2) \leq \frac{[p_1/\ln(1-p_1)][p_2/\ln(1-p_2)]}
 {\sqrt{(1 - p_1)(1 - p_2)}} \left(2 - \frac{\pi^2}{6}\right)
 - \sqrt{(1 - p_1)(1 - p_2)}.
\]
\end{lemma}

\begin{proof}
The minimum correlation between
$X_1$ and $X_2$ with $X_1 \sim \geo(p_1)$ and $X_2 \sim \geo(p_2)$
is determined by $\mean[X_1 X_2]$ and is found when
$X_1 = \chi_1(U)$ and 
$X_2 =\chi_2(U)$ (where $U \sim \unif([0,1])$).
Hence
\[
\mean[X_1 X_2] = \int_0^1 \left\lfloor \frac{\ln(1 - u)}{\ln(1 - p_1)} 
 \right\rfloor \left\lfloor \frac{\ln(u)}{\ln(1 - p_2)} \right\rfloor\ du 
\]

For any nonnegative $a$ and $b$, $\lfloor ab \rfloor \leq ab$, so 
\begin{align*}
\mean[X_1 X_2 ] &\leq
 \int_0^1 \frac{\ln(1 - u)\ln(u)}{\ln(1 - p_1)\ln(1 - p_2)} 
 &= [\ln(1 - p_1)\ln(1 - p_2)]^{-1}(2 - \pi^2/6)
\end{align*}
where $\int_0^1 \ln(1-u)\ln(u) \ du$ can be computed by considering
the power series expansion of $\ln(1 - u)$ and the value
for the Riemann zeta function at 2 (see for example \citealp{dukicm2013}).

When $X \sim \geo(p)$, $\mean[X] = (1 - p)/p$ and 
$\var(X) = \mean[X]^2/(1 - p)$.  Hence
\begin{align*}
\rho_{-}(p_1,p_2) &\leq \frac{[\ln(1 - p_1)\ln(1-p_2)]^{-1}
  (2 - \pi^2/6) - \mean[X_1]\mean[X_2]}
  {\mean[X_1]\mean[X_2]/\sqrt{(1 - p_1)(1 - p_2)}}.
\end{align*}
Simplifying then finishes the proof.
\end{proof}

The following lemma gives a feel for the behavior of $-p/\ln(1 - p)$.

\begin{lemma}
For $p \in (0,1/2]$,
\[
1 - (2 - \ln(2)^{-1}) p \leq \frac{-p}{\ln(1 - p)} \leq 
 1 - (1/2)p - (1/12)p^2.
\]
where $2 - \ln(2)^{-1} \approx 0.5573$.
\end{lemma}

To obtain a lower bound, first note, as in \cite{dukicm2013}, that 
\[
\frac{\int_0^1 \lambda_1^{-1}\lambda_2^{-1}
  \ln(u)\ln(1 - u) \ du - \lambda_1^{-1}\lambda_2^{-1}}
 {\lambda_1^{-1}\lambda_2^{-1}}
 = 1 - \pi^2/6 = -0.6449\ldots
\]
is the minimum correlation between any two exponentially distributed
random variables, no matter their rates!

It is well known that adding an exponential random variable of rate
$\lambda$ conditioned to lie in $[0,1]$
to a geometric with parameter $p = 1 - \exp(-\lambda)$
gives an exponential random variable with rate $\lambda$.  This can
be used to show the following.

\begin{lemma}
Let 
\[
g(p_1,p_2) =
 \frac{[p_1/\ln(1-p_1)][p_2/\ln(1-p_2)]}
 {\sqrt{(1 - p_1)(1 - p_2)}} \left(2 - \frac{\pi^2}{6}\right)
 - \sqrt{(1 - p_1)(1 - p_2)}. 
\]
The minimum correlation satisfies
\[
g(p_1,p_2) - \frac{1}{2}\sqrt{\frac{1-p_1}{1-p_2}} p_2
 - \frac{1}{2} \sqrt{\frac{1-p_2}{1-p_1}} p_1 
  \leq \rho_{-}(p_1,p_2) \leq g(p_1,p_2)
\]
\end{lemma}

\begin{proof}
For $X_1 \sim \geo(p_1),\ X_2 \sim \geo(p_2),$ let
\[
 A_1 \sim \expdist(-\ln(1 - p_1)|A_1 \in [0,1]),\ 
 A_2 \sim \expdist(-\ln(1 - p_2)|A_2 \in [0,1]),
\]
where $A_1$ and $A_2$ are independent of $(X_1,X_2)$ and each other.
Then $X_i + A_i \sim \expdist(-\ln(1-p_i))$ for $i \in \{1,2\}$, and 
so $\cor(X_1 + A_1,X_2 + A_2) \geq 1 - \pi^2/6.$

Solving the correlation for the mean of the product gives:
\[
\mean[(X_1 + A_1)(X_2 + A_2)] \geq (2 - \pi^2/6)\ln(1-p_1)\ln(1-p_2).
\]
So
\[
\mean[X_1 X_2] \geq -\mean(A_1)\mean(X_2) - \mean(A_2)\mean(X_1)
 - \mean(A_1)\mean(A_2) + 
 (2 - \pi^2/6)\ln(1-p_1)\ln(1-p_2).
\]
Since $\mean(A_1)$ and $\mean(A_2)$ are both at most $1/2$, this
gives
\[
\mean[X_1 X_2] \geq - (1/2)\mean(X_2) - (1/2)\mean(X_1)
 - 1/4 + 
 (2 - \pi^2/6)\ln(1-p_1)\ln(1-p_2).
\]
which in turn gives the result.
\end{proof}

Theorem~\ref{THM:second} then follows easily.

\section{Appendix}
\label{SEC:appendix}

Here we carry out
in greater detail the calculation of $\rho_{-}(p)$ $(p_1=p_2=p)$ 
that is used to generate Figure \ref{FIG:graph}.
 
 Consider $1/2 \in [p,1-p]$.
Let $k$ be such that $\alpha_k \leq 1/2 <\alpha_{k+1}$. 
That implies $q^{k+1}<1/2 \leq q^k$ and 
$k \leq \log_{q}(1/2) < k+1$.  Since $k$ is an integer,
$k = \lfloor \log_q(1/2)\rfloor.$

To avoid accumulation of superscripts let $r_i = (1/2)^{1/i}$, 
the $i$th root of $1/2$. Then $r_i \leq q < r_{i+1}$ gives $k=i$, 
so as a function of $q$, $k$ is a step-function whose value  increases by one at the roots of 1/2.
%We will assume $q$ is not any root of $1/2$.\\

\begin{figure}[h]
\begin{tikzpicture}
\draw[line width=1.5pt] (0,0) -- (13, 0) ;

\draw (0, -0.1) -- (0, 0.1) node [above] {$\alpha_1$};
\draw (1, -0.1) -- (1, 0.1) node [above] {$\beta_{c_1}$};
\draw (2.2, -0.1) -- (2.2, 0.1) node [above] {$\beta_{c_1-1 }$};
 \draw (2.9,0.2) node [above] {$\ldots$};
 \draw (3.7, -0.1) -- (3.7, 0.1) node [above] {$\beta_{c_2 + 1 }$};
\draw (4.9, -0.1) -- (4.9, 0.1) node [above] {$\beta_{c_2 }$};
%\draw (9, -0.1) -- (9, 0.1) node [above] {$\beta_{k+2}$};

\draw (9.7, -0.1) -- (9.7, 0.1) node [above] {$\beta_{k+1}$};

\draw (11.3,-0.1)--(11.3,0.1) node [above] {$\frac{1}{2}$};
\draw (12.1,-0.1)--(12.1,0.1) node [above] {$\beta_k$};
%second line
\draw[line width=1.5pt] (0,-1) -- (13, -1);
\draw (0, -0.9) -- (0, -1.1) node [below] {$\alpha_1$};
\draw (4.3, -0.9) -- (4.3, -1.1) node [below] {$\alpha_2$};
%\draw (7, -0.9) -- (7, -1.1) node [below] {$\alpha_3$};
\draw (8,-1) node [below] {$\cdots$};
\draw (10.5, -0.9) -- (10.5, -1.1) node [below] {$\alpha_k$};
\draw (11.3,-0.9)--(11.3,-1.1) node [below] {$\frac{1}{2}$};
\draw (12.9, -0.9) -- (12.9, -1.1) node [below] {$\alpha_{k+1}$};
%vertical dashed
\draw[dotted] (1,0)--(1,-2.1);
\draw[dotted] (2.2,0)--(2.2,-2.1);
\draw[dotted] (3.7,0)--(3.7,-2.1);
\draw[dashed] (4.3,0)--(4.3,-2.1);

\draw[dotted] (4.9,0)--(4.9,-2.1);
\draw[dotted] (9.7,0)--(9.7,-2.1);
\draw[dashed] (10.5,0)--(10.5,-2.1);
\draw (11.3,0)--(11.3,-2.1);
%\draw[dashed] (11.8,0)--(11.8,-2.1);
%third line
\draw[line width=1pt] (0,-2) -- (13, -2);
%\draw (-1.6,-2.2) node[below] {\scriptsize $F^{-1}(U)F^{-1}(1 - U) = $};
\draw (0.4,-2.2) node[below] {\scriptsize {$1\cdot c_1$}};
\draw (1.6,-2.2) node[ below] {\scriptsize{$ 1(c_1-1)$}};
\draw (3,-2.2) node[below] {\scriptsize{$\cdots$}};
\draw (4,-2.2) node[below] {\scriptsize{$ 1  c_2$}};
\draw (4.6,-2.2) node[below] {\scriptsize{$2 c_2$}};
\draw (5.6,-2.2) node[below] {\scriptsize{$2 (c_2-1)$}};
\draw (8,-2.2) node[below] {\scriptsize{$\cdots$}};
\draw (10,-2.2) node[below] {\scriptsize{$(k-1) k$}};
\draw (10.9,-2.2) node[below] {\scriptsize{$k \cdot k$}};
\end{tikzpicture}
\caption{$\alpha_i$, $\beta_i$, $k$, and $c_i$ over $[p=\alpha_1,1/2]$ and
slightly beyond.  The
bottom row represents the value of $F_p^{-1}(U)F_p^{-1}(1 - U)$ in
each subinterval.}
\label{FIG:twolines}
\end{figure}
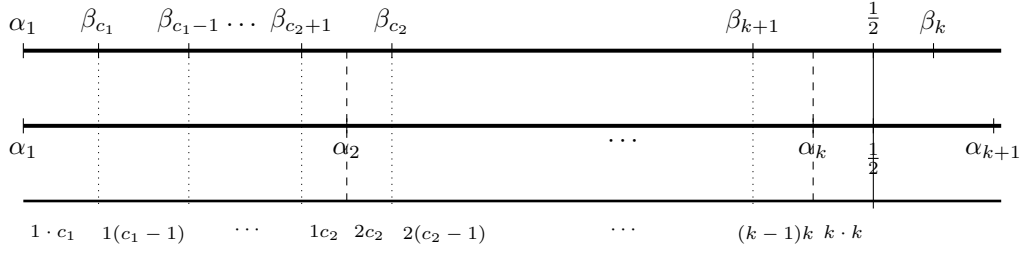

%\begin{eqnarray}
%\sum_{i=1}^{k-1} \big{(} F^{-1}(U) F^{-1}(1-U) \ind (U \in (\alpha_i,\beta_{c_i})) + \sum_{}^{}
%\end{eqnarray}
%%ovo sredi 
%Half of the expectation is

For $i \in \{1,\ldots,c\}$, let 
$c_i$ be the index such that 
$\beta_{c_i+1}< \alpha_i \leq \beta_{c_i}$. 
Then $c_i=[\log_{q}(1-q^i)]$ (see Figure \ref{FIG:twolines}.)
%We can prove that $c_k = k+1$. 
The mean product of a geometric and its antithetic counterpart can
be written
\begin{align} \label{sum1}
\mean(F_p^{-1}(U) F_p^{-1}(1-U)) = 2 \sum_{i=1}^{k} i L_i
\end{align}
where for $i=1,\ldots,k-1$
\begin{align*}
L_i=  \Big{(} |(\alpha_i, \beta_{c_i})| c_i + 
  |(\beta_{c_{i}},\beta_{c_i -1})| (c_i -1) + \cdots 
+ |(\beta_{c_{i+1}+1},\alpha_{i+1})| c_{i+1} \Big{)}.
\end{align*}
[Here $|(a,b)| = b - a$ denotes the width of the interval.]

When $i = k$ there are three cases
\begin{enumerate}
\item[Case 1.] $\beta_{k+1} \leq \alpha_k \leq 1/2$, so $L_k=  |(\alpha_k,\frac{1}{2})| k^2$ 
since $c_k = k$.  Then
\begin{align*}%\label{case1}
L_k= k^2(q^k - 1/2).
\end{align*}
\item[Case 2.] $\beta_{k+2} \leq \alpha_k \leq  \beta_{k+1}  \leq 1/2$ so
$L_k=|(\alpha_k,\beta_{k+1})| k (k+1) +|(\beta_{k+1},\frac{1}{2})| k^2$
since $c_k = k + 1$. 
Then
\begin{align*}%\label{case2}
L_k=k^2(q^k-1/2)+k(q^{k+1}-1+q^k).
\end{align*}
\item[Case 3.] $\alpha_k \leq \beta_{k+2}<\beta_{k+1} \leq 1/2$:  Here it is the case that 
$c_k=k+2$ and
$L_k=|(\alpha_k,\beta_{k+2})| k (k+2) + |(\beta_{k+2},\beta_{k+1})| k(k+1) +|(\beta_{k+1},\frac{1}{2})| k^2.$   Then
\begin{align*}%\label{case3}
L_k= k^2(q^k - 1/2)+k(q^{k+2}+ q^{k+1}+2q^{k}-2).
\end{align*}
\end{enumerate} 
These three cases exhaust the possibilities.

\begin{lemma}
The set of $\beta_i$ values in $[\alpha_k,1/2]$ is either $\emptyset$,
$\{\beta_{k+1}\}$, or $\{\beta_{k+1},\beta_{k+2}\}$.
\end{lemma}

\begin{proof} It suffices to show that 
$\beta_{k+3}< \alpha_k$ which is equivalent to
$q^{k+3} < 1- q^k$. As before, let $r_i = (1/2)^{1/i}$.  
Consider the function $g(x) = x^{i+3}+x^i-1$ on the interval
$(r_i, r_{i+1})$; we shall show that $g(x)$ 
is negative there. At $r_{i+1}$:
\begin{align*}
%g(r_{i+1}) = (1/2)(r_{i+1}^2+r_{i+1}^{-1}-2)= 1/(2r_{i+1}) (r_{i+1}^3-2r_{i+1} + 1).
g(r_{i+1}) &= (1/2)^{(i + 3)/(i+1)} + 
             (1/2)^{i/(i+1)}  - 1  \\
 &= (1/2)[r_{i+1}^2 + r_{i+1}^{-1} - 2] = (2 r_{i+1})^{-1}[r_{i+1}^3 + 1 - 2r_{i+1}].
\end{align*}
Now we observe that $x^3-2x+1< 0$ for $x \geq r_1$ 
and therefore $g(r_{i+1}) < 0$. 
Since $g'(x) = {i+3}x^{i+2} + i x^{i-1} > 0$, $g$ is an increasing function on 
$(r_i, r_{i+1})$ which means the the function is negative on the entire 
interval.

So between $\alpha_k$ and $1/2$ one finds either
$\beta_{k+1}$, \{$\beta_{k+1} , \beta_{k+2}$\} or no $\beta_i$ values.
\end{proof}

%We will show also that between every pair of consecutive roots of 1/2, namely in every interval $(r_i, r_{i+1})$, all three situations happen. Basic calculus.\\
% 
% {\large At which points exactly, that is the problem! Intersecting points at the picture below Figure \ref{alphabeta}}.
% \begin{figure}[h]
% \includegraphics[width=5in]{alphabetapic.pdf} \label{alphabeta} 
% \caption{x-axis = q; Red=$1-\alpha^k$; Blue= $\beta^{k+1}$; Yellow = $\beta^{k+2}$} 
% \end{figure}
 
\vskip5mm
\par Going back to \eqref{sum1}, for $i=1,\ldots,k-1$ we have
\begin{align*}
L_i&= \Big{(} (\beta_{c_i} -\alpha_i) c_i +  
      (\beta_{c_i - 1} - \beta_{c_i}) (c_i - 1) + \cdots + {} \\
&\hspace*{1.5em} (\beta_{c_{i+1}+1}-\beta_{c_{i+1}+2}) (c_{i+1}+1)+(\alpha_{i+1}-\beta_{c_{i+1}+1})c_{i+1}\Big{)} \\
 &= c_i (\beta_{c_i}-\alpha_i + \beta_{c_i-1}-\beta_{c_i} + \cdots +
    \alpha_{i+1}-\beta_{c_{i+1}+1}) - 1\cdot(\beta_{c_i-1} - \beta_{c_i}) - {} \\
 &\hspace*{1.5em} 2\cdot(\beta_{c_i -2}-\beta_{c_i-1}) - \cdots 
  - (c_i-c_{i+1})(\alpha_{i+1}-\beta_{c_{i+1}+1})
 \end{align*}
 %sad ovde izracunaj zadnji red
which in terms of $q$ is
 \begin{align*}
L_i &= c_i(\alpha_{i+1}-\alpha_i) - 
    (c_i-c_{i+1}) \alpha_{i+1}-\beta_{c_{i+1}+1}(\beta_{c_i-c_{i+1}} + 
    \beta_{c_i-c_{i+1}-1} + \cdots + 1) \\
      &= c_i( q^i - 1) -c_{i+1}(q^{i+1}-1)+ \frac{q^{c_{i+1}+1}-q^{c_i+1}}{1-q}.
\end{align*}
We can rewrite the sum in \eqref{sum1} as
\begin{align} \label{sumL}
\sum_{i=1}^{k-1} i L_i = \sum_{i=1}^{k-1}  L_i +\sum_{i=2}^{k-1} L_i+ \cdots +  
\sum_{i=k-2}^{k-1}  L_i+ L_{k-1}.
\end{align}
Many terms in $\sum_{i=1}^{k-1} L_i$ cancel:
\begin{align*}
\sum_{i=1}^{k-1}  L_i &= c_1(q-1)- c_2(q^2-1) + 
                       \frac{q^{c_{2}+1}-q^{c_1+1}}{1-q} + {}\\
     &\hspace*{1.5em} c_2(q^2-1)- c_3(q^3-1) +
      \frac{q^{c_{3}+1}-q^{c_2+1}}{1-q} + \cdots +  \\
     &\hspace*{1.5em} 
       c_{k-1}(q^{k-1}-1)- c_k(q^k-1) +\frac{q^{c_{k}+1}-q^{c_{k-1}+1}}{1-q}  \\
     &= c_1(q-1)- c_k(q^k-1) + \frac{q^{c_{2k}+1}-q^{c_1+1}}{1-q}.
\end{align*}
In general, when the sum starts at $j$:
\begin{align*}
 \sum_{i=j}^{k-1} L_i = c_j(q^j-1)- c_k(q^k-1) + \frac{q}{1-q}(q^{c_k}-q^{c_j}); ~~j=1,2,\ldots,k-1.
 \end{align*}
 Now the sum in \eqref{sumL} becomes
\begin{align*}
\sum_{i=1}^{k-1} i L_i &= 
  \sum_{i=1}^{k-1} c_i(q^i-1)- c_k(q^k-1) + \frac{q}{1-q}(q^{c_k}-q^{c_i})\\
  &= -(k-1)c_k (q^k-1) + (k-1)\frac{q^{c_k+1}}{1-q} + 
     \sum_{i=1}^{k-1} c_i(q^i-1)- \frac{q^{c_i+1}}{1-q}.
\end{align*}

Finally
\begin{align*}
\frac{1}{2}\mean(F_p^{-1}(U) F_p^{-1}(1-U)) &= \sum_{i=1}^{k-1} c_i (q^i - 1) -  
  \frac{q^{c_i+1}}{1-q} + R(q,k), 
\end{align*}
where $R(q,k)$ equals
\begin{equation*}
\begin{array}{ll}
  k^2/2 +k(q^k-1+ q^{k+1}/(1-q)) - q^{k+1}/(1-q) & \quad 
     \text{in case 1,} \\
  k^2/2+k(q^{k+1}+q^k-1+q^{k+2}/(1-q)+q^k-1-q^{k+2}/(1-q)) & 
    \quad \text{in case 2,} \\
  k^2/2+k(q^{k+2}+ q^{k+1}+q^k -1 + q^{k+3}/(1-q))+2(q^k-1) - 
    q^{k+3}/(1-q) & \quad \text{in case 3.}
  \end{array}
\end{equation*}

%\begin{flalign}
% \frac{1}{2}\mean(F_p^{-1}(U) F_p^{-1}(1-U)) =  \sum_{i=1}^{k-1} c_i (q^i - 1) %-  \frac{q^{c_i+1}}{1-q}+ \hspace{40mm} \nonumber \\[3mm]
% +\left\{ 
%  \begin{array}{l l}
%    1/2 k^2 +k(q^k-1+ \frac{q^{k+1}}{1-q}) - \frac{q^{k+1}}{1-q} & \quad \text{%Case 1.}\\
%    1/2k^2+k(q^{k+1}+q^k-1+\frac{q^{k+2}}{1-q})+q^k-1-\frac{q^{k+2}}{1-q} & \qu%ad \text{Case 2.}\\
%     1/2k^2+k(q^{k+2}+ q^{k+1}+q^k -1 +\frac{q^{k+3}}{1-q})+2(q^k-1)-\frac{q^{k%+3}}{1-q} & \quad \text{Case 3.}
%  \end{array} \right.
%\end{flalign}

%\section*{Acknowledgements}
%The author would like to thank an anonymous referee for
%pointing out several mistakes in a preliminary version.

\bibliographystyle{alea3}
\bibliography{alea1}

\end{document}